\documentclass[12pt]{article}
\usepackage{lmodern}
\usepackage{microtype}
\usepackage[english]{babel}
\usepackage[all]{xy}
\usepackage{fancyhdr}
\usepackage{setspace, amsmath, amsthm, amssymb, amsfonts, amscd, epic, graphicx, ulem, dsfont}
\usepackage{amsthm}
\usepackage[T1]{fontenc}
\usepackage{cases}
\newtheorem{theorem}{Theorem}

\newtheorem{proposition}[theorem]{Proposition}
\newtheorem{corollary}[theorem]{Corollary}
\newtheorem{example}[theorem]{Example}

\makeatletter

\@addtoreset{equation}{section}
\makeatother

\title{On the non-existence of symphonic maps}

\author{Ahmed Mohammed Cherif\footnote{University Mustapha Stambouli Mascara, Faculty of Exact Sciences, Mascara 29000, Algeria. Email: a.mohammedcherif@univ-mascara.dz} and Kaddour Zegga\footnote{University Mustapha Stambouli Mascara, Faculty of Exact Sciences, Mascara 29000, Algeria. Email: zegga.kadour@univ-mascara.dz}}
\date{}

\begin{document}
\maketitle
	
\begin{abstract}
In this paper, we study the existence of symphonic  maps on compact or complet non-compact Riemannian manifold into Riemannian manifolds admitting a conformal vector field or a non-trivial Ricci solitons.
\end{abstract}

\section{Introduction}

Liouville-type theorems for symphonic maps typically assert that under certain conditions, a symphonic map must be trivial (e.g., constant). Symphonic maps are critical points of variational problems, such as the energy functional, under variations that preserve volume. These maps generalize harmonic maps and arise in the context of various geometric and physical settings.\\
Let $(M^{m}, g)$, $(N^{n}, h)$ be Riemannian manifolds without boundary, and let $\varphi$ be a smooth map from $M^m$ into $N^n$.  Let $\varphi^{*}h$ be the pullback of the metric $h$ by $\varphi$, i.e.,
$$( \varphi^{*}h)(X, Y ) = h(d \varphi (X), d \varphi (Y )),$$
for any vector fields $X$, $Y$ on $M^m$. We consider the functional or the symphonic energy
$$E_{sym}(\varphi)=\int_{M^m}|\varphi^{*}h|^{2}dv_{g},$$
where $dv_{g}$ is the volume form on $(M^m, g)$, and  $|\varphi^{*}h|$ denotes the norm of the pullback  $\varphi^{*}h$, i.e.,
 $$|\varphi^{*}h|^{2}=h(d\varphi(e_{i}),d\varphi(e_{j}))^{2},$$  and
 $\{e_{i}\}_{i=1}^m$ is a local orthonormal frame on $(M^{m}, g)$.
 The symphonic energy $E_{sym}(\varphi)$ is related to the energy $E(\varphi)$ in the theory of harmonic maps since the functional $E_{sym}(\varphi)$ is an integral of the norm of the pullback  $\varphi^{*}h$, while the energy $E(\varphi)$ is an integral of the trace of the pullback $\varphi^{*}h$. Indeed, the energy $E(\varphi )$ is defined to be
$$ E(\varphi)=\int_{M^m}|d\varphi|^{2}dv_{g},$$
where $|d\varphi|^{2} =h(d\varphi(e_{i}), d\varphi(e_{i}))$.  A map $\varphi$ is called harmonic if it is a critical point of the functional energy $E(\varphi)$, i.e., if the first variation of $E(\varphi)$ at $\varphi$ vanishes.
If $M$ is noncompact, $\varphi$ is defined to be a harmonic map if it is a critical point of the functional energy $E(\varphi)$ on any compact subdomain of $M^m$. The theory of harmonic maps is developed with applications to other fields \cite{EL,ES}. The first variation formula of the  functional energy  $E_{sym}(\varphi)$  was given by Shigeo Kawaia, Nobumitsu Nakauchi in \cite{KN}.

\begin{proposition}[First Variation Formula \cite{KN}]
We have
$$\frac{dE_{sym}(\varphi_{t})}{dt}\Big|_{t=0}=-4\int_{M}h(\operatorname{div}_{g}\sigma_{\varphi},\upsilon)dv_{g},$$
for any variation vector field $\upsilon$, where  $\sigma_{\varphi}$ is defined by
$$\sigma_{\varphi}(X)=h(d\varphi(X),d\varphi(e_{j}))d\varphi(e_{j}), \quad  \forall  X \in \Gamma(TM), $$ and  the divergence of $\sigma_{\varphi}$ is given  by
\begin{eqnarray}\label{EQ1}
\nonumber \operatorname{div}_{g}\sigma_{\varphi}
&=& h(\nabla d\varphi(e_{i},e_{i}),d\varphi(e_{j}))d\varphi(e_{j}) + h(d\varphi(e_{i}),\nabla d\varphi(e_{i},e_{j}))d\varphi(e_{j})\\
&& +h(d\varphi(e_{i}),d\varphi(e_{j}))\nabla d\varphi(e_{i},e_{j}).
\end{eqnarray}
\end{proposition}
Recall that the map $\varphi$ is called symphonic if it is a critical point of the functional energy
 $E_{sym}$.
We denote by  $\tau^{s}(\varphi)=\operatorname{div}_{g}\sigma_{\varphi}$. Thus, the map $\varphi$ is symphonic if and only if  $\tau^{s}(\varphi)=0$.\\
A smooth vector field $\xi$ on a Riemannian manifold $(M^m,g)$ is said to a conformal vector field if there exists a smooth function $f$ on $M^m$ that satisfies
\begin{eqnarray}\label{E11}
\mathfrak{L}_{\xi} g=2f g,
\end{eqnarray} where $\mathfrak{L}_{\xi} g$ is the Lie derivative of $g$ with respect $\xi$, that is the flow of the vector field $\xi$ consists of conformal transformations of the Riemannian manifold $(M^m,g)$, the function $f$ is called the potential function of the conformal vector field $\xi$. Conformal vector fields generalize Killing vector fields, which are infinitesimal isometries that preserve the metric exactly (i.e., for which $f=0$). Conformal vector fields are important in theoretical physics, particularly in the context of general relativity and conformal field theories, where conformal transformations play a central role. We say $\xi$ a nontrivial conformal vector field if $\xi$  is a non-Killing conformal vector field  (for more detail see \cite{OBT,TAW,TASH}).\\
A Ricci soliton structure on a Riemannian manifold $(M^m,g)$ is the choice of a
smooth vector field $\xi$ satisfying the soliton equation
\begin{equation}\label{eq1.6}
   \operatorname{Ric}^{M^m}+\frac{1}{2}\mathfrak{L}_\xi g=\lambda g,
\end{equation}
for some constant $\lambda\in\mathbb{R}$, where $\operatorname{Ric}^{M^m}$ is the Ricci curvature of $(M^m,g)$.
The Ricci soliton $(M^m,g,\xi,\lambda)$ is said to be shrinking,
steady or expansive according to whether the coefficient $\lambda$ appearing in equation (\ref{eq1.6}) satisfies
$\lambda > 0$, $\lambda = 0$ or $\lambda < 0$ (see \cite{RH,H2}). If $(M^m, g)$ is not Einstein, we call the Ricci soliton non-trivial.\\
In \cite{cherif3}, the author studied the existence of harmonic maps into
Riemannian manifolds admitting a conformal vector field, or
a non-trivial Ricci soliton. In this paper, we study the case of symphonic maps into Riemannian manifolds admitting
such special vector fields with some assumptions.

\section{Main result}

\subsection{Symphonic map  and conformal vector fields}
\begin{proposition}\label{comp}
Let $(M^m,g)$ be a compact connected orientable Riemannian manifold without boundary, and let $(N^n,h)$ be a Riemannian manifold admitting a conformal vector field $\xi$ with potential function $f > 0$. Then, any symphonic map $\varphi:
(M^m,g)\longrightarrow(N^n,h)$ is constant.
\end{proposition}

\begin{proof}
Let $X\in \Gamma(TM)$, and let $\{e_i\}_{i=1}^m$ be a normal orthonormal frame at $x \in M^m$. We set
$$\omega(X)=h(\xi \circ \varphi,\sigma_{\varphi}(X)),$$
Then, we have at $x$
\begin{eqnarray}\nonumber
\operatorname{div}^{M^m}\omega &=& h\big(\nabla^{\varphi}_{e_{i}}\xi \circ \varphi,\sigma_{\varphi}(e_{i})\big)
+h(\xi \circ \varphi,\nabla^{\varphi}_{e_{i}}\sigma_{\varphi}(e_{i})).
\end{eqnarray}
Using the fact that $\varphi$ is symphonic and $\xi$ is a conformal vector field in $(N^n,h)$, we find that
\begin{equation}
\operatorname{div}^{M^m}\omega = (f\circ \varphi)h(d\varphi(e_{i}),d\varphi(e_{j}))^{2}.
\end{equation}
Integrating over the compact manifold $M^m$, we obtain
\begin{equation}
\int_{M^m}(\operatorname{div}^{M^m}\omega) dv_{g} =\int_{M^m} (f\circ \varphi)h(d\varphi(e_{i}),d\varphi(e_{j}))^{2} dv_{g}.
\end{equation}
Using the divergence Theorem (see \cite{BW}), and the fact that $$h(d\varphi(e_{i}),d\varphi(e_{j}))^{2}\geq h(d\varphi(e_{i}),d\varphi(e_{i}))^{2},$$ we have
$$0\geq \int_{M^m} (f\circ \varphi)h(d\varphi(e_{i}),d\varphi(e_{i}))^{2} dv_{g},$$
thus for all $i=1,\dots,m$, we have $h(d\varphi(e_{i}),d\varphi(e_{i}))=0$, so $\varphi$ is constant.
\end{proof}

\begin{example}
The Euclidian space $(\mathbb{R}^n,h=dy_1^2+...+dy_n^2)$ (resp. the hyperbolic space $(\mathbb{H}^n,h=y_n^{-2}(dy_1^2+...+dy_n^2))$) has a conformal vector field $\xi$ with potential function $f(y)\neq0$ at any point $y\in\mathbb{R}^n$ (resp. $y\in\mathbb{H}^n$)
given by $\xi=y_i\frac{\partial}{\partial y_i},$  where $f(y)=1$ for all $y\in\mathbb{R}^n$
(resp. $\xi=\sum_{i=1}^{n-1}y_i\frac{\partial}{\partial y_i}+(y_n-1)\frac{\partial}{\partial y_n},$
where $f(y)=y_n^{-1}$ for all $y\in\mathbb{H}^n$). Hence,
any symphonic map $\varphi$ from compact connected orientable Riemannian manifold without boundary into $\mathbb{R}^n$ or $\mathbb{H}^n$
is constant map.
\end{example}

From Proposition \ref{comp}, we derive the following result.

\begin{corollary}\label{C1}
Let $(\overline{N}^n,\overline{h})$ be a Riemannian manifold that admits a Killing vector field $\overline{\xi}$. Consider $(N^{n-1},h)$ as a Riemannian hypersurface of $(\overline{N}^n,\overline{h})$, where $h$ is the induced Riemannian metric of $\overline{h}$ on $N^{n-1}$. Suppose that

i) $(N^{n-1},h)$ is totally umbilical, i.e.,
$$B(X,Y) = \rho h(X,Y)\eta, \quad \forall \, X,Y \in \Gamma(TN),$$
for some smooth function $\rho$ on $N^{n-1}$, where $B$ is the second fundamental form of $N^{n-1}$ in $\overline{N}^n$, and $\eta$ is the unit normal vector field to $N^{n-1}$ in $\overline{N}^{n}$.

ii) The function $\overline{h}(\overline{\xi}, H) \neq 0$ everywhere on $N^{n-1}$, where $H = \frac{1}{n-1}\operatorname{trace}_{h} B$ is the mean curvature of $(N^{n-1},h)$ in $(\overline{N}^{n},\overline{h})$.

Then, any symphonic map from a compact connected orientable Riemannian manifold without boundary to $(N^{n-1},h)$ is constant.
\end{corollary}

\begin{proof}
We have $\overline{\xi} = \xi + f\eta$, where $\xi$ is the tangential component of $\overline{\xi}$ along $N^{n-1}$, and $f$ is a smooth function on $N^{n-1}$. Since $\overline{\xi}$ is a Killing vector field and $(N^{n-1},h)$ is totally umbilical, we get (see \cite{cherif3})
\begin{equation}\label{conf}
(\mathfrak{L}_{\xi}h)(X,Y) = 2f\rho h(X,Y),
\end{equation}
The Corollary follows from Proposition \ref{comp}, and equation (\ref{conf}), with the potential function
$$f \rho = \overline{h}(\overline{\xi}, \eta)\overline{h}(H,\eta) = \overline{h}(\overline{\xi}, H),$$
which is non-zero everywhere on $N^{n-1}$ by assumption ii).
\end{proof}

\begin{example}
 We consider the hypersurface $$N^{n-1}=\mathbb{S}^{n-1}\cap \{y \in \mathbb{R}^n \,|\, y_n > 0 \},$$ in $ \overline{N}^n = \mathbb{R}^n$  equipped with the standard Riemannian metric $\overline{h} = <,>_{\mathbb{R}^n} $, where $\mathbb{S}^{n-1}$
is the unit $(n-1$)-dimensional sphere on $\mathbb{R}^n$. Let $h$ the induced Riemannian metric on $N$. We have\\
i) \   $B(X, Y ) = -h(X, Y )P$  for all $X, Y \in \Gamma(T N)$, where $P$ is the position vector field of $\mathbb{R}^n$, then  $(N^{n-1}, h)$ is totally umbilical, with  $\rho = 1 \ and \  \eta = -P $ along $N^{n-1}$.\\
ii) Moreover, we have  $$h(\frac{\partial}{\partial y_n},H)= -y_n \neq 0,$$
everywhere on $N^{n-1}$ (see \cite{cherif3}). Since $\overline{\xi} = \frac{\partial}{\partial y_n}$ is a Killing vector field on ($\overline{N}^n, \overline{h})$. According to the Corollary \ref{C1}, any symphonic map from a compact connected orientable Riemannian manifold without boundary to $(N^{n-1},h)$ is constant.
 \end{example}


In the case of complete non-compact Riemannian manifold, we obtain the following results.

\begin{proposition}\label{noncomp}
Let $(M^m,g)$ be a complete non-compact Riemannian manifold,  $(N^n,h)$ a Riemannian manifold admitting a conformal vector field $\xi$ with potential function $f > 0$, and  $\varphi : (M^m,g) \longrightarrow (N^n,h)$ a symphonic map. We assume that
\begin{equation}
\int_{M^m}\frac{|\xi \circ \varphi|^{2} |d\varphi|^{2}}{f \circ \varphi} dv_{g} < \infty.
\end{equation}
Then $\varphi$ is constant.
\end{proposition}

\begin{proof}
Let $\rho$ be a smooth function with compact support on $M^m$, we set
$$\omega(X)=h(\xi \circ \varphi,\rho^{2}\sigma_{\varphi}(X)), \quad \forall X\in \Gamma(TM).$$
We compute the divergence of $\omega$
\begin{eqnarray}
\operatorname{div}^{M^m}\omega
&=&\nonumber h\big(\nabla^{\varphi}_{e_{i}}\xi \circ \varphi,\rho^{2}\sigma_{\varphi}(e_{i})\big)+ 2\rho e_{i}(\rho)h(\xi \circ \varphi,\sigma_{\varphi}(e_{i}))\\
&&+ \rho^{2}h(\xi \circ \varphi,\nabla^{\varphi}_{e_{i}}\sigma_{\varphi}(e_{i})),
\end{eqnarray}
where $\{e_i\}_{i=1}^m$ is a local orthonormal frame of vector fields on $M^m$.
Since $\varphi$ is symphonic map, and $\xi$ is conformal vector field, we find that
\begin{equation}\label{div}
\operatorname{div}^{M^m}\omega = \rho^{2}(f\circ \varphi)h(d\varphi(e_{i}),d\varphi(e_{j}))^{2}+ 2\rho e_{i}(\rho)h(\xi \circ \varphi,\sigma_{\varphi}(e_{i})).
\end{equation}
By Young's inequality, for any $\lambda > 0$, we have
$$-2\rho e_{i}(\rho)h(\xi \circ \varphi,\sigma_{\varphi}(e_{i}))\leq \lambda e_{i}(\rho)^{2}| \xi\circ\varphi |^{2}+\frac{\rho^{2}}{\lambda} | \sigma_{\varphi}(e_{i}) |^{2}. $$
Combining with (\ref{div}), we get
\begin{eqnarray*}
   \rho^{2}(f\circ \varphi)h(d\varphi(e_{i}),d\varphi(e_{j}))^{2}-\operatorname{div}^{M^m}\omega
   &\leq& \lambda e_{i}(\rho)^{2}| \xi \circ \varphi |^{2} \\
   &&  +\frac{\rho^{2}}{\lambda}h(d\varphi(e_{i}),d\varphi(e_{j}))^{2} | d\varphi(e_{j}) |^{2}.
\end{eqnarray*}
Take $\lambda=\frac{2 \vert d\varphi \vert^{2}}{f \circ \varphi}$ on $M^m_+=\{x\in M^m\,,\,\vert d\varphi \vert>0\}$.
As $| d\varphi(e_{j}) |^{2}\leq | d\varphi |^{2}$, $\forall j=\overline{1,m}$, we deduce the inequality
 \begin{equation}\label{youn}
 \frac{1}{2}\rho^{2}(f \circ \varphi)h(d\varphi(e_{i}),d\varphi(e_{j}))^{2}-\operatorname{div}^{M^m}\omega \leq \frac{2\vert d\varphi \vert^{2}}{f \circ \varphi} e_{i}(\rho)^{2}| \xi \circ \varphi |^{2}.
 \end{equation}
 By the divergence Theorem, and (\ref{youn}) we have
\begin{equation}\label{ineg}
\frac{1}{2}\int_{M^m}\rho^{2}(f \circ \varphi)h(d\varphi(e_{i}),d\varphi(e_{j}))^{2}dv_{g}
\leq 2\int_{M^m}\frac{\vert d\varphi \vert^{2}| \xi \circ \varphi |^{2}}{f \circ \varphi} e_{i}(\rho)^{2}dv_{g}.
\end{equation}
Consider the smooth cut-off function $\rho = \rho_{R}$, such that $\rho \leq 1$ on $M^m$, $\rho = 1$ on the geodesic ball $B(p,R)$, $\rho = 0$ on $M^m \setminus B(p,2R)$, and $| \operatorname{grad}^{M^m}\rho | \leq \frac{2}{R}$. From (\ref{ineg}), we get
 \begin{equation}
 \frac{1}{2}\int_{B(p,R)}(f \circ \varphi)h(d\varphi(e_{i}),d\varphi(e_{j}))^{2}dv_{g}
\leq \frac{8}{R^{2}}\int_{B(p,2R)}\frac{\vert d\varphi \vert^{2}| \xi \circ \varphi |^{2}}{f \circ \varphi} dv_{g}.
 \end{equation}
Since $\int_{M^m}\frac{|\xi\circ \varphi|^{2}|d\varphi|^{2}}{f\circ \varphi}dv_{g} < \infty$
 by assumption, letting $R \rightarrow \infty$ yields
$$\int_{M^m}\rho^{2}(f\circ \varphi)h(d\varphi(e_{i}), d\varphi(e_{j}))^{2}dv_{g} = 0.$$
 Consequently, $\varphi$ is constant.
\end{proof}

Since the identity map is symphonic on any Riemannian manifold,
according to Proposition \ref{noncomp}, we get the following Corollary.

\begin{corollary}\cite{cherif3}
Let $(M^m,g)$ be a complete non-compact Riemannian manifold, and let $\xi$ a conformal vector field on
$(M^m,g)$ with potential function $f > 0$. Then, the integral of $\frac{| \xi |^{2}}{f}$ over $M^m$ is infinity.
\end{corollary}

\begin{example}
Let $(N^{n-1},h)$ be a Riemannian manifold. The warped product Riemannian manifold $(\mathbb{R}\times_{\lambda} N^{n-1},dt^2+\lambda(t)^2h)$
(where $\lambda$ is a smooth positive function on $\mathbb{R}$) admits $\xi=\lambda(t)\partial_t$
as a conformal vector field with potential function $f(t,y)=\lambda'(t)$ for all $(t,y)\in\mathbb{R}\times N^{n-1}$,
where $\lambda'$ is the derivative function of $\lambda$ (see \cite{BYC}). We suppose that $\lambda'>0$.
According to Proposition \ref{noncomp}, the integral of $\lambda(t)^2/\lambda'(t)$ over $\mathbb{R}\times_{\lambda} N^{n-1}$ is infinity, or  $\mathbb{R}\times_{\lambda} N^{n-1}$ is not complete.
\end{example}


\subsection{Symphonic maps to Ricci solitons}

\begin{proposition}\label{solyt}
Let $(M^m,g)$ be a compact connected orientable Riemannian manifold without boundary, and $(N^n,h,\xi ,\lambda )$ a non-trivial Ricci soliton. We assume that
\begin{equation}\label{ricc}
\operatorname{Ric}^{N^n} > \lambda h  \quad\hbox{or}\quad \operatorname{Ric}^{N^n} < \lambda h.
\end{equation}
Then, any symphonic map $\varphi:(M^m,g)\longrightarrow(N^n,h)$ is constant.
\end{proposition}

\begin{proof}
Let  $\omega\in\Gamma(T^*M)$ defined by
$$\omega(X)=h(\xi \circ \varphi,\sigma_{\varphi}(X)),\quad\forall X\in \Gamma(TM).$$
Let $\{e_i\}_{i=1}^m$ be a normal orthonormal frame at $x \in M^m$. We have at $x$
\begin{equation}\label{divw1}
\operatorname{div}^{M^m}\omega =e_{i}h\big(\xi \circ \varphi,\sigma_{\varphi}(e_{i})\big),
\end{equation}
by equation (\ref{divw1}), the definition of $\sigma_{\varphi}$, and the symphonic condition of $\varphi$, we get the following
\begin{equation}\label{divw2}
\operatorname{div}^{M^m}\omega =h\big(\nabla^{\varphi}_{e_{i}}\xi \circ \varphi,\sigma_{\varphi}(e_{i})\big)=\frac{1}{2}(\mathfrak{L}_{\xi}h)(d\varphi(e_i),\sigma_{\varphi}(e_{i})),
\end{equation}
from the Ricci soliton equation, we find that
\begin{equation}\label{divw3}
\operatorname{div}^{M^m}\omega =\lambda h\big(d\varphi(e_{i}),\sigma_{\varphi}(e_{i})\big)
-\operatorname{Ric}^{N^n}(d\varphi(e_i),\sigma_{\varphi}(e_{i})),
\end{equation}
by replacing the  expression of $\sigma_{\varphi}(e_i)$ \ in  (\ref{divw3}), we obtain
\begin{equation}\label{divw4}
\operatorname{div}^{M^m}\omega =h\big(d\varphi(e_i),d\varphi(e_{j})\big)\big[\lambda h\big(d\varphi(e_{i}),d\varphi(e_{j})\big)-\operatorname{Ric}^{N^n}(d\varphi(e_i),\varphi(e_{j}))\big].
\end{equation}
We can assume that $\varphi^*h$ is diagonalizable at $x$, so the equation (\ref{divw4}) becomes
\begin{equation}\label{divwf}
\operatorname{div}^{M^m}\omega =h\big(d\varphi(e_i),d\varphi(e_{i})\big)\big[\lambda h\big(d\varphi(e_{i}),d\varphi(e_{i})\big)-\operatorname{Ric}^{N^n}(d\varphi(e_i),\varphi(e_{i}))\big].
\end{equation}
The Proposition \ref{solyt} follows from equation (\ref{divwf}), the assumption (\ref{ricc}), and the divergence Theorem.
\end{proof}

\begin{example}
It is known that the cigar soliton
$\big(\mathbb{R}^{2}, \frac{dx^2 + dy^2}{1+ x^2 + y^2} \big)$
is steady with strictly positive Ricci tensor (see \cite{RH}), according to Proposition \ref{solyt},
any symphonic map $\varphi$ from a compact connected orientable Riemannian manifold without boundary to the cigar
soliton is constant.
\end{example}

\begin{example}
Let $N^{n+1} = \mathbb{R}\times \mathbb{S}^{n}$ equipped with the diagonal Riemannian metric $\overline{h} = dt^{2} +h$, where
$h$ is the induced Riemannian metric on the sphere $\mathbb{S}^{n}$, and let
$f(t,x)=at^{2}+bt +c$ for all $(t,x)\in \mathbb{R}\times \mathbb{S}^{n},$
where $a, b, c \in \mathbb{R}$. Then, $(N^{n+1},\overline{h}, \operatorname{grad}f, \lambda)$ is a Ricci soliton of gradient type, that is $\xi=\operatorname{grad}f$, with $\lambda=n-1$ and $a=(n-1)/2$ (see \cite{mekki}). The Ricci curvature of $(N^{n+1},\overline{h})$ is given by
$\operatorname{Ric}^{N^{n+1}}=(n-1)\overline{h}-(n-1)dt^{2}$, hence $\operatorname{Ric}^{N^{n+1}}< \lambda \overline{h}$. According to Proposition (\ref{solyt}), any symphonic map $\varphi$ from a compact connected orientable Riemannian manifold without boundary to $(N^{n+1},\overline{h}, \operatorname{grad}f, \lambda)$ is constant.
\end{example}

In the case of complete non-compact Riemannian manifold, we obtain the following results.

\begin{theorem}\label{comric}
Let $(M^m,g)$ be a complete non-compact Riemannian manifold, and \ $(N^n,h,\xi ,\lambda )$ \ a non-trivial Ricci soliton such that \ $\operatorname{Ric}^N < \mu h$ for some constant $\mu < \lambda$. \ If
$\varphi : (M^m,g)\longrightarrow (N^n,h)$ is symphonic map, satisfying
$$\int_{M^m}| \xi \circ \varphi|^{2} | d\varphi |^{2}v_{g}< \infty,$$
then $\varphi$ is constant.
\end{theorem}

\begin{proof}
Let $\rho$ be a smooth function with compact support on $M$. Define
$$\omega(X)=h(\xi \circ \varphi,\rho^{2}\sigma_{\varphi}(X)), \quad \forall X\in \Gamma(TM).$$
Given a normal orthonormal frame $\{e_i\}_{i=1}^m$ at $x \in M^m$, we have at $x$
\begin{equation}\label{divw6}
\operatorname{div}^{M^m}\omega =e_{i}h\big(\xi \circ \varphi,\rho^{2}\sigma_{\varphi}(e_{i})\big),
\end{equation}
From Equation \eqref{divw6}, and the symphonic condition of $\varphi$, we obtain
\begin{equation}\label{divw7}\begin{aligned}
\operatorname{div}^{M^m}\omega &=h\big(\nabla^{\varphi}_{e_i}\xi \circ \varphi,\rho^{2}\sigma_{\varphi}(e_{i})\big)+
h\big(\xi \circ \varphi,\nabla^{\varphi}_{e_i}\rho^{2}\sigma_{\varphi}(e_{i})\big)\\
&= \rho^{2} h\big(\nabla^{\varphi}_{e_i}\xi \circ \varphi,\sigma_{\varphi}(e_{i})\big)+2 \rho e_i(\rho)
h(\xi \circ \varphi,\sigma_{\varphi}(e_i)).
\end{aligned}\end{equation}
Using the Ricci soliton equation, and the definition of $\sigma_{\varphi}$, we find that
\begin{equation}
\rho^{2} h\big(\nabla^{\varphi}_{e_i}\xi \circ \varphi,\sigma_{\varphi}(e_{i})\big)=\lambda \rho^{2} h\big( d\varphi(e_i),\sigma_{\varphi}(e_{i})\big)-\rho^{2}\operatorname{Ric}^{N^n}(d\varphi(e_i),\sigma_{\varphi}(e_{i})).
\end{equation}
Thus, equation \eqref{divw7} becomes
\begin{equation}\label{divw8}\begin{aligned}
\operatorname{div}^{M^m}\omega & = \lambda \rho^{2} h\big( d\varphi(e_i),\sigma_{\varphi}(e_{i})\big)-\rho^{2}\operatorname{Ric}^{N^n}(d\varphi(e_i),\sigma_{\varphi}(e_{i}))\\
&\quad+2 \rho e_i(\rho)h(\xi \circ \varphi,\sigma_{\varphi}(e_i)).
\end{aligned}\end{equation}
By using Young's inequality, we have
\begin{equation}\label{young}
-2\rho e_{i}(\rho)h(\xi \circ \varphi,\sigma_{\varphi}(e_{i}))\leq \frac{1}{\epsilon} e_{i}(\rho)^{2}| \xi\circ\varphi |^{2}+ \epsilon \rho^{2} | \sigma_{\varphi}(e_{i}) |^{2},
\end{equation}
for all $\epsilon > 0$. From equation \eqref{divw8}, and inequality \eqref{young}, we deduce
\begin{eqnarray}\label{ineg1}
\lambda \rho^{2} h\big( d\varphi(e_i),\sigma_{\varphi}(e_{i})\big)-\rho^{2}\operatorname{Ric}^{N^n}(d\varphi(e_i),\sigma_{\varphi}(e_{i}))-\operatorname{div}^{M^m}\omega &\leq&\nonumber \frac{1}{\epsilon} e_{i}(\rho)^{2}| \xi\circ\varphi |^{2}\\
&&\nonumber+ \epsilon \rho^{2} | \sigma_{\varphi}(e_{i}) |^{2}.\\
\end{eqnarray}
Since $\sigma_{\varphi}(e_{i})=h(d\varphi(e_i),d\varphi(e_j))d\varphi(e_j)$  and $\forall j=\overline{1,m}, \quad | d\varphi(e_{j}) |^{2}\leq | d\varphi |^{2}$, from inequality \eqref{ineg1}, we get
\begin{eqnarray}\label{ineg2}
\lambda \rho^{2} h\big( d\varphi(e_i),d\varphi(e_{j})\big)^2   &-&\nonumber\rho^{2}h\big(d\varphi(e_i),d\varphi(e_{j})\big)\operatorname{Ric}^{N^n}(d\varphi(e_i),d\varphi(e_{j}))\\
&\leq&\nonumber \operatorname{div}^{M^m}\omega +\frac{1}{\epsilon} e_{i}(\rho)^{2}| \xi\circ \varphi |^{2}\\
& & + \epsilon \rho^{2}h\big( d\varphi(e_i),d\varphi (e_{j})\big)^2 | d\varphi |^{2}.
\end{eqnarray}
Letting $\epsilon =\frac{\lambda - \mu }{ | d\varphi |^{2}}$ on $M^m_+=\{x\in M^m\,,\,\vert d\varphi \vert>0\}$, we obtain from \eqref{ineg2}
\begin{eqnarray}\label{ineg3}
 \rho^{2}h\big(d\varphi(e_i),d\varphi(e_{j})\big)\Big[\mu h\big( d\varphi(e_i),d\varphi(e_{j})\big)
   &-& \operatorname{Ric}^{N^n}(d\varphi(e_i),d\varphi(e_{j}))\Big] \\
   &\leq&\nonumber \operatorname{div}^{M^m}\omega +\frac{| d\varphi |^{2}}{\lambda - \mu} e_{i}(\rho)^{2}| \xi \circ \varphi |^{2}.
\end{eqnarray}
By using the divergence Theorem, and inequality \eqref{ineg3}, we have
\begin{equation}\label{ineg4}\begin{aligned}
 \int_{M^m}\rho^{2}h\big(d\varphi(e_i),d\varphi(e_{j})\big)\Big[\mu h\big( d\varphi(e_i),d\varphi(e_{j})\big)&-\operatorname{Ric}^{N^n}(d\varphi(e_i),d\varphi(e_{j}))\Big]dv_g   \\
& \leq\frac{1}{\lambda - \mu}\int_{M^m} e_{i}(\rho)^{2}| d\varphi |^{2}| \xi \circ \varphi |^{2}dv_g,
\end{aligned} \end{equation}
Consider the cut-off smooth function $\rho = \rho_{R}$ such that $\rho \leq 1$ on $M$, $\rho = 1$ on the ball $B(p,R)$, $\rho = 0$ on $M^m \setminus B(p,2R)$, and $| \operatorname{grad}^{M^m}\rho | \leq \frac{2}{R}$. From \eqref{ineg4}, we get
\begin{equation}\label{ineg5}\begin{aligned}
 \int_{B(p,R)}h\big(d\varphi(e_i),d\varphi(e_{j})\big)\Big[\mu h\big( d\varphi(e_i),d\varphi(e_{j})\big)&-\operatorname{Ric}^{N^n}(d\varphi(e_i),d\varphi(e_{j}))\Big]dv_g   \\
&\leq\frac{4}{(\lambda - \mu)R^2}\int_{B(p,2R)}| d\varphi |^{2}| \xi \circ \varphi |^{2}dv_g,
\end{aligned} \end{equation}
 Since $\int_{M^m}| \xi \circ \varphi|^{2} | d\varphi |^{2}v_{g}< \infty$, when $R \rightarrow +\infty$, we obtain
 $$ \int_{M^m}h\big(d\varphi(e_i),d\varphi(e_{j})\big)\Big[\mu h\big( d\varphi(e_i),d\varphi(e_{j})\big)-\operatorname{Ric}^{N^n}(d\varphi(e_i),d\varphi(e_{j}))\Big]dv_g=0.$$
 Recall that  $ \mu h-\operatorname{Ric}^N > 0  $, we conclude that
  $d\varphi(e_i) = 0$ for all $i=\overline{1,m}$, that is $\varphi$ is constant.
\end{proof}


\subsection*{Conflict of interest}
The author declares no conflict of interest.

\subsection*{Data availability}
Not applicable.

\subsection*{Funding Declaration}
No funding was received.

\end{document}